\documentclass[11pt,a4paper,reqno]{amsart}
\usepackage{amsmath}
\usepackage{amsfonts}
\usepackage{amssymb}
\usepackage{subcaption}
\numberwithin{equation}{section}
\theoremstyle{plain}

\usepackage[utf8]{inputenc}
\usepackage[T1]{fontenc}
\usepackage{ae,aecompl}
\usepackage{amsmath}
\usepackage{amssymb}
\usepackage{amsthm}
\usepackage{array}
\usepackage{xy}
\usepackage{bm}
\usepackage{graphicx}
\date{}

\DeclareMathOperator{\conv}{conv}
\DeclareMathOperator{\lin}{lin}

\DeclareMathOperator{\vol}{vol}

\DeclareMathOperator{\inte}{int}

\newtheorem{twr}{Theorem}
\newtheorem{lem}[twr]{Lemma}

\theoremstyle{remark}
\newtheorem{remark}[twr]{Remark}

\makeatletter
\renewcommand\@seccntformat[1]{\csname the#1\endcsname.\quad}

\makeatother

\begin{document}

\title[Extremal Banach-Mazur distance on the plane]{Extremal Banach-Mazur distance between a symmetric convex body and an arbitrary convex body on the plane}

\author{Tomasz Kobos}
\thanks{The research of the author was supported by Polish National Science Centre grant 2014/15/N/ST1/02137}

\address{Faculty of Mathematics and Computer Science \\ Jagiellonian University \\ Lojasiewicza 6, 30-348 Krakow, Poland}

\email{Tomasz.Kobos@im.uj.edu.pl}


\begin{abstract}
We prove that if $K, L \subset \mathbb{R}^2$ are convex bodies such that $L$ is symmetric and the Banach-Mazur distance between $K$ and $L$ is equal to $2$, then $K$ is a triangle. 
\end{abstract}

\maketitle

\section{Introduction}

A set $K \subset \mathbb{R}^n$ is called a \emph{convex body} if it is compact, convex and with non-empty interior. A convex body $K$ is \emph{centrally-symmetric} (or just \emph{symmetric}) if it has a center of symmetry. By $\partial K$ we will denote the boundary of a convex body $K$. For a vector $z \in \mathbb{R}^n$ let $K_z = K-z$ be a shift of $K$. If $0$ is an interior point of $K$, then we define the \emph{polar body} $K^{\circ}$ of $K$, as
$$K^{\circ} = \{ x \in \mathbb{R}^n \ : \ \langle x, y \rangle \leq 1 \text{ for every } y \in K \},$$
where $\langle \cdot, \cdot \rangle$ denotes the usual scalar product in $\mathbb{R}^n$. 

The \emph{Banach-Mazur distance} is a fundamental notion of the geometry of Banach spaces and of the convex geometry. It is a multiplicative distance between normed spaces of the same dimension or a multiplicative distance between arbitrary convex bodies of the same dimension. We shall work in the context of convex bodies and in this case it is defined as
$$d(K, L) = \inf \{r>0: K + u \subset T(L+v) \subset r(K+u)\}.$$
where the infimum is taken over all invertible linear operators $T: \mathbb{R}^n \to \mathbb{R}^n$ and $u, v \in \mathbb{R}^n$.

As the Banach-Mazur distance has many interesting and important applications in the fields of convex and discrete geometry and also in functional analysis, this notion has been studied extensively for years by many authors. From the famous John's ellipsoid theorem (see \cite{john}) it immediately follows that $d(K, L) \leq n$ for all symmetric convex bodies $K, L \subset \mathbb{R}^n$. Gluskin in \cite{gluskin}, using a sophisticated probabilistic construction, was able to show that this bound has asymptotically the right order: there exists a constant $c>0$ such that, for every integer $n \geq 1$, there are centrally symmetric convex bodies $K_n, L_n \subset \mathbb{R}^n$ satisfying $d(K_n, L_n) \geq c n$. The exact value of the diameter for the symmetric case is known only for $n=2$ and it is equal to $\frac{3}{2}$ by a result of Stromquist (see \cite{stromquist}).

Much less is known in the non-symmetric setting. John's ellipsoid theorem yields the estimate $d(K, L) \leq n^2$ for all convex bodies $K, L \subset \mathbb{R}^n$. Rudelson in \cite{rudelson} has proved that the asymptotic order of this bound can be improved a lot. The exact value of the diameter is not known even for $n=2$, but Lassak proved that $d(K, L) \leq 3$ for all convex bodies $K, L \subset \mathbb{R}^2$ (see \cite{lassakdiameter}). This estimate was recently improved by Brodiuk, Palko and Prymak to $\frac{19-\sqrt{73}}{4} \approx 2.614$ (see \cite{brodiuk}). The best known lower bound is only linear in $n$. Challenges that are faced when dealing with the Banach-Mazur distance are well illustrated by the fact, that the exact distance between the three-dimensional cube and the octahedron is not known. It seems very likely and is supported by computational results that this distance is equal to $\frac{9}{5}$ (see \cite{xue}), but the proof is yet to be found.

In light of these difficulties, it may come as a surprise that it is possible to determine explicitly the maximal possible distance between a symmetric convex body and an arbitrary convex body in every dimension $n$. In 1963 Gr\"unbaum introduced a following variation of the Banach-Mazur distance for arbitrary convex bodies $K, L \subset \mathbb{R}^n$: 
$$d_G(K, L) = \inf \{ |r| \ : \ K' \subset L' \subset rK' \}$$ 
with the infimum taken over all non-degenerate affine images $K'$ and $L'$ of $K$ and $L$ respectively. Therefore, instead of sandwiching the affine copy of $L$ between two positive homothets of $K$, we may use also the negative homothets as well. It is clear that $d_G(K, L) = d(K, L)$, if $K$ or $L$ is symmetric. Gr\"unbaum has conjectured that the maximal possible distance is equal to $n$. It was confirmed more than $40$ years later by Gordon, Litvak, Meyer and Pajor (see \cite{gordon}), who gave a proof based on the \emph{John's decomposition in the general case}, that shall be discussed in the next section of the paper. It is also not hard to prove that if $L$ is a symmetric convex body and $S_n$ is a simplex in $\mathbb{R}^n$, then $d_G(L, S_n)=d(L, S_n)=n$. A natural question arises: is it true that if $L$ is symmetric and $d(K, L)=n$, then is $L$ a simplex? Or even more generally, if $d_G(K, L)=n$ for $K, L$ arbitrary, then is one of $K$ and $L$ a simplex? The more general question is a conjecture of Jim\'{e}nez and Nasz\'{o}di stated in \cite{jimenez}. These authors have managed to establish a set of conditions that has to be satisfied in the case of the equality $d_G(K, L)=n$. Based on them, they proved that if $L$ is strictly convex or smooth (and is not necessarily symmetric) and $d_G(K, L)=n$, then $K$ is a simplex in $\mathbb{R}^n$. Their result is a broad generalization of a result of Leichtweiss who proved it in a special case of $L$ being an Euclidean ball in $\mathbb{R}^n$ (see \cite{leichtweiss}). A different proof was provided much later by Palmon in \cite{palmon}.

The goal of the paper is to provide a proof of the planar case of the conjecture of Jim\'{e}nez and Nasz\'{o}di under the assumption that $L$ is symmetric. Our main result goes as follows

\begin{twr}
\label{twglowne}
Let $K$ and $L$ be convex bodies in the plane such that $L$ is symmetric. If $d(K, L)=2$, then $K$ is a triangle. 
\end{twr}

We were not able to prove in full generality that if $K, L$ are arbitrary and $d_G(K, L)=2$, then one of $K$ and $L$ is a triangle. However, we have managed to prove this in many different configurations, depending on the John's decomposition for $K$ and $L$. The full statement of our main result is quite technical and is presented in details at the beginning of the proof. We remark that, besides the result of Jim\'{e}nez and Nasz\'{o}di for strictly convex and smooth bodies, the only single case of our theorem that we know that was established before, is the case of $L$ being a parallelogram. It was proved by Lassak in \cite{lassakpar}.

The paper is organized as follows. In Section \ref{johnpoz} we introduce the John's decomposition in general case, that is the starting point of the proof of Theorem \ref{twglowne}. Proof of our main result is presented Section \ref{dowody}. In the last section of the paper we provide some concluding remarks.

\section{John's position of convex bodies}
\label{johnpoz}

Let $K, L \subset \mathbb{R}^n$ be arbitrary convex bodies. We say that $K$ is in a \emph{position of maximal volume} in $L$ if $K \subset L$ and for all affine images $K'$ of $K$ contained in $L$ we have $\vol(K') \leq \vol(K)$. A classical technique used for upper-bounding the Banach-Mazur distance between two convex bodies, is to consider a position of maximal volume. The result of Gordon, Litvak, Meyer and Pajor provides us a set of precise conditions on so-called \emph{contact pairs} when we consider a position of maximal volume. Their result is a broad generalization of the original line of reasoning of John in the Euclidean case. If $K, L \subset \mathbb{R}^n$ are convex bodies we say that $K$ is in \emph{John's position} in $L$ if $K \subset L$, and
\begin{itemize}
\item $x = \sum_{i=1}^{m} a_i \langle x, u_i \rangle v_i$ for every $x \in \mathbb{R}^n$,
\item $0 =  \sum_{i=1}^{m} a_i u_i =  \sum_{i=1}^{m} a_i v_i,$
\item $\langle u_i, v_i \rangle = 1$ for $1 \leq i \leq m$,
\end{itemize}
for some integer $m>0$, $\{u_i: 1 \leq i \leq m \} \subset \partial K \cap \partial L$, $\{v_i: 1 \leq i \leq m \} \subset \partial K^{\circ} \cap \partial L^{\circ}$ and positive $a_i$'s. We shall call $(u_i, v_i)$ the \emph{contact pairs}. It is easy to prove that these conditions imply that $a_1+a_2 + \ldots + a_m = n$.

These two positions of convex bodies are related by the following result:
\begin{twr}[Gordon, Litvak, Meyer, Pajor \cite{gordon}]
\label{decomposition}
If $K, L \subset \mathbb{R}^n$ are convex bodies such that $K$ is in a position of maximal volume in $L$ and $0 \in \inte K$, then there exists $z \in \frac{n}{n+1}K$ such that $K- z$ is in John's position in $L-z$ with $m \leq n^2 + n$.
\end{twr}

Based on the John's decomposition theorem, Gordon, Litvak, Meyer and Pajor were able to give a very short and elegant proof of the following result.

\begin{twr}[Gordon, Litvak, Meyer, Pajor \cite{gordon}]
\label{glmp}
Let $K, L \subset \mathbb{R}^n$ be convex bodies such that $K \subset L$ are in a John's position. Then $L \subset - nK$.
\end{twr}

By combining these two results above one gets the inequality $d_G(K, L) \leq n$ for arbitrary convex bodies $K, L \subset \mathbb{R}^n$, which confirms the conjecture of Gr\"unbaum.

\section{Proof of Theorem \ref{twglowne}}
\label{dowody}

To characterize the equality in the inequality $d_G(K, L) \leq n$ for $L$ strictly convex or smooth, Jim\'{e}nez and Nasz\'{o}di have carefully followed the estimates of Gordon, Litvak, Meyer and Pajor made in the proof of Theorem \ref{glmp}, to derive a set of conditions that have to hold in the case of equality. We state them in the form of a lemma.

\begin{lem}
\label{conditions}
Let $K, L \subset \mathbb{R}^n$ be convex bodies such that $K \subset L$ are in a John's position with contact pairs $(u_i, v_i)_{i=1}^{m}$. Assume that $x \in \partial L \cap (-n \partial K)$. Let $w \in \partial K^{\circ}$ be any vector such that $\langle x, w \rangle = -n$. Let $A = \{i: \langle u_i, w \rangle < 1 \}$ and $B = \{i: \langle u_i, w \rangle = 1 \}$ (thus $A \cup B = \{1 ,2, \ldots, m\}$). Then we have
\begin{equation}
\label{convu}
\frac{-x}{n} \in \conv \{ u_i : i \in B \},
\end{equation}
\begin{equation}
\label{convv}
\frac{-w}{n} \in \conv \{ v_i : i \in A \},
\end{equation}
\begin{equation}
\label{xv}
\langle x, v_i \rangle = 1 \text{ for all } i \in A.
\end{equation}

\end{lem}

\begin{proof}
See Proposition $10$ in \cite{jimenez}.
\end{proof}

The following two lemmas also come directly from the paper of Jim\'{e}nez and Nasz\'{o}di. Especially the second lemma will be of special importance to us.

\begin{lem}
\label{johnconvex}
Let $K, L \subset \mathbb{R}^n$ be convex bodies such that $K \subset L$ are in a John's position. Then, the contact pairs $(u_i, v_i)_{i=1}^{m}$, appearing in the definition of John's position (see Section \ref{johnpoz}), can be choosed in such a way that
$$u_i \not \in \conv \{ u_1,  \ldots, u_{i-1}, u_{i+1}, \ldots, u_m \} \text{ and } v_i \not \in \conv \{ v_1, \ldots, v_{i-1}, v_{i+1}, \ldots, v_m \},$$ 
for every $1 \leq i \leq m$.
\end{lem}

\begin{proof}
See Proposition $10$ in \cite{jimenez}.
\end{proof}

\begin{lem}
\label{convex}
Let $K, L \subset \mathbb{R}^n$ be convex bodies such that $K \subset L$, $0 \in \inte L$ and $K \not \subset rL+v$ for any $0 < r < 1$ and $v \in \mathbb{R}^n$. Then $0 \in \conv(\partial K^{\circ} \cap \partial L^{\circ}).$
\end{lem}

\begin{proof}
See Lemma $11$ in \cite{jimenez}.
\end{proof}

Note that in the John's position $K \subset L$ there have to be at least $n+1$ contact pairs $(u_i, v_i)_{i=1}^{n+1}$. The situation in which there are exactly $n+1$ of them is quite special.

\begin{lem}
\label{n+1}
Suppose that $K, L \subset \mathbb{R}^n$ are convex bodies such that $K \subset L$ are in a John's position and there are exactly $n+1$ contact pairs $(u_i, v_i)_{i=1}^{n+1}$. Then $a_1=a_2 \ldots = a_{n+1} = \frac{n}{n+1}$ and $\langle u_i, v_j \rangle = -\frac{1}{n}$ for $i \neq j$.
\end{lem}

\begin{proof}
Let $a_1, a_2, \ldots, a_m$ be the corresponding weights, defined as in Section \ref{johnpoz}. The matrix with columns $u_1, u_2, \ldots u_{n+1}$ has the rank $n$ as $\lin \{u_1, u_2, \ldots, u_{n+1}\}= \mathbb{R}^n$. Thus if the equality $\sum_{i=1}^{n+1} c_iu_i=0$ holds, then $c_i = ta_i$ for some $t \in \mathbb{R}$ and every $1 \leq i \leq m$. For any $x \in \mathbb{R}^n$ we have
$$x = \sum_{i=1}^{n+1} a_i \langle x, v_i \rangle u_i.$$
By plugging $x=u_j$ for some $1 \leq j \leq n+1$, we get
$$(a_j - 1)u_j + \sum_{i \neq j} a_i \langle u_j, v_i \rangle u_i = 0,$$
so that $\langle u_j, v_i \rangle = t = \frac{a_j-1}{a_j}.$ for $i \neq j$. By repeating the same argument for $v_i$, we get that $\langle u_j, v_i \rangle = \frac{a_i-1}{a_i}$ for $i \neq j$. This proves that all numbers $a_1, a_2, \ldots, a_{n+1}$ are equal. Since their sum is equal to $n$, we have $a_1 = a_2 = \ldots = a_{n+1} = \frac{n}{n+1}$. It follows that
$$\langle u_i, v_j \rangle = \frac{\frac{n}{n+1}-1}{\frac{n}{n+1}}=-\frac{1}{n}$$ 
and the lemma is proved.

\end{proof}

Before proceeding with the proof of Theorem \ref{twglowne}, we shall describe shortly our approach. A classical result of real analysis states that if $f: \mathbb{R} \to \mathbb{R}$ is a differentiable function and $x_0$ is a local maximum of $f$, then $f'(x_0)=0$. We can think of Theorem \ref{decomposition} as a very involved analogue of this well-known fact, taking place in a much more complicated setting. It gives us a set of necessary conditions that have to hold for a maximum of a mapping $T \to \det(T)$, under the constraint $T(K) \subset L$. If we would have a similar set of conditions not only for a maximum volume position, but also for a position that minimizes the Banach-Mazur or Gr\"unbaum distance, then we can imagine that our state of knowledge in this field would be much greater. It seems that it can be very difficult to describe the minimal distance position of convex bodies in full generality. However, in the particular case of the equality $d_G(K, L)=2$, if we use conditions given in Lemma \ref{conditions}, then in most situations, we can prove by hand that we are not in minimum distance position. In other words, if $K, L$ are different from a triangle, then in most configurations, we can show directly how to find a small affine perturbation $L'$ of $L$ such that $K \subset L' \subset \inte (-2K)$.

\emph{Proof of Theorem \ref{twglowne}}. Let us suppose that $K \subset L$ is in a position of maximal volume. Then, by Theorem \ref{decomposition} for some $z \in \inte K$, convex bodies $K-z \subset L-z$ are in a John's position. Without loss of generality, we may suppose that $z=0$. By Lemma \ref{johnconvex} we can further assume that contact pairs $(u_i, v_i)_{i=1}^{m}$ satisfy
\begin{equation}
\label{convzalozenie}
u_i \not \in \conv \{ u_1, \ldots, u_{i-1}, u_{i+1}, \ldots, u_m \} \text{ and } v_i \not \in \conv \{ v_1,\ldots, v_{i-1}, v_{i+1}, \ldots, v_m \},
\end{equation}
for every $1 \leq i \leq m$. Then from Theorem \ref{glmp} it follows that $L \subset -2K$. Clearly, the set $\partial L \cap \partial (-2K)$ contains some element $x$. Directly from Theorem \ref{decomposition} we have that $3 \leq m \leq 6$. In our case however, if $m$ would be at least $5$, then one of the sets $A, B$, defined as in Lemma \ref{conditions} for $x$, would have to contain at least $3$ elements. If $A$ would contain at least $3$ elements, then by condition (\ref{xv}) some $3$ of the vectors $v_1, v_2, \ldots, v_m$ would lie on one line, contradicting the assumption (\ref{convzalozenie}). Similarly, if $B$ would contain at least $3$ elements, then according to its definition, some $3$ of the vectors $u_1, u_2, \ldots, u_m$ would lie on one line, also contradicting assumption (\ref{convzalozenie}). In this way we have proved that $m=3$ or $m=4$. 

As already mentioned in the introduction, we are going to prove a stronger statement than in the formulation of Theorem. Let $s$ be the cardinality of the set $\partial L \cap \partial (-2K)$ (possibly $s = \infty$). We shall prove that
\begin{enumerate}
\item If $m=3$, then the equality $d_G(K, L)=2$ implies that $K$ is a triangle or $L$ is a triangle, with $L$ not necessarily symmetric.
\item If $m=4$ and $s \neq 3$, then the equality $d_G(K, L)=2$ implies that $K$ is a triangle or $L$ is a triangle, with $L$ not necessarily symmetric (we shall see later that in this case we must have $s=2$ or $s=4$).
\item If $m=4$ and $s=3$ and $L$ is additionaly symmetric, then the equality $d_G(K, L)=2$ implies that $K$ is a triangle.
\end{enumerate}
Thus, the only case in which we have to assume that $L$ is symmetric is the situation of four contact pairs in John's position $K \subset L$ with exactly three contact points of the boundaries of $L$ and $-2K$. 

\textbf{Case (1)}. We can suppose that the contact points $u_1, u_2, u_3$ form an equilateral triangle. By Lemma \ref{n+1} its center of the gravity is in $0$ and the lines $u_2u_3, u_3u_1, u_1u_2$ are parallel to supporting lines at $u_1, u_2, u_3$ respectively. For a simplicity, let us denote $a=-2u_1, b=-2u_2, c=-2u_3$. Then we have a chain of inclusions
$$\conv\{ u_1, u_2, u_3 \} \subset K \subset L \subset \conv \{a, b, c \} \subset -2K \subset \conv \{ 4u_1, 4u_2, 4u_3 \}.$$

In particular every contact point of $\partial L$ and $\partial (-2K)$ belongs to the perimeter of the triangle $abc$. We shall consider different possibilities for the contact points of $\partial L$ and $\partial (-2K)$. 

\textbf{Case (1a)}: all contact points of $\partial L$ and $\partial (-2K)$ lie in the interiors of the sides of the triangle $abc$. If $x$ is such a contact point, lying in the interior of the side $ab$ then, clearly the whole segment $[a, b]$ is contained in $\partial (-2K)$. Moreover, the only supporting line to $-2K$ at $x$ is the line $ab$. Therefore, by Lemma \ref{convex} each side has to contain at least one contact point. But then $-2K$ contains the segments $[ab], [bc], [ca]$ in the boundary and hence $-2K=\conv \{a, b, c \}$. In this case the conclusion follows.

\textbf{Case (1b)}: there is exactly one contact point of $\partial L$ and $\partial (-2K)$ that belongs to the set $\{a, b, c\}$. Suppose that it is $c$. Note that there is some point of contact $x \in \partial L \cap \partial (-2K)$ that belongs to the side $ab$. Otherwise, the convex hull of $\partial L^{\circ} \cap \partial (-2K)^{\circ}$ would not contain $0$, which would contradict Lemma \ref{convex}. Suppose that $x$ belongs to the interior of the side $ab$. Then $[a, b] \subset \partial (-2K)$ and also $[u_1, u_2] \subset \partial K$. In particular, $u_3 \in \partial L^{\circ} \cap \partial (-2K)^{\circ}$, but possibly there are more points of contact along the side $[a, b]$. The triangle $\conv \{u_1, u_2, c \}$ is contained in $L$. Furthermore, if both segments $[u_1, u_3], [u_2, u_3]$ are contained in the boundary of $K$, then clearly $K = \conv \{u_1, u_2, u_3\}$. Let us first consider the case in which both of these segments are \textbf{not} contained in the boundary of $K$. Figure \ref{gj3k1} demonstrates a position of convex bodies $K \subset L \subset -2K$ in this situation.

\begin{figure}[h]
\centering
\includegraphics[width=\textwidth]{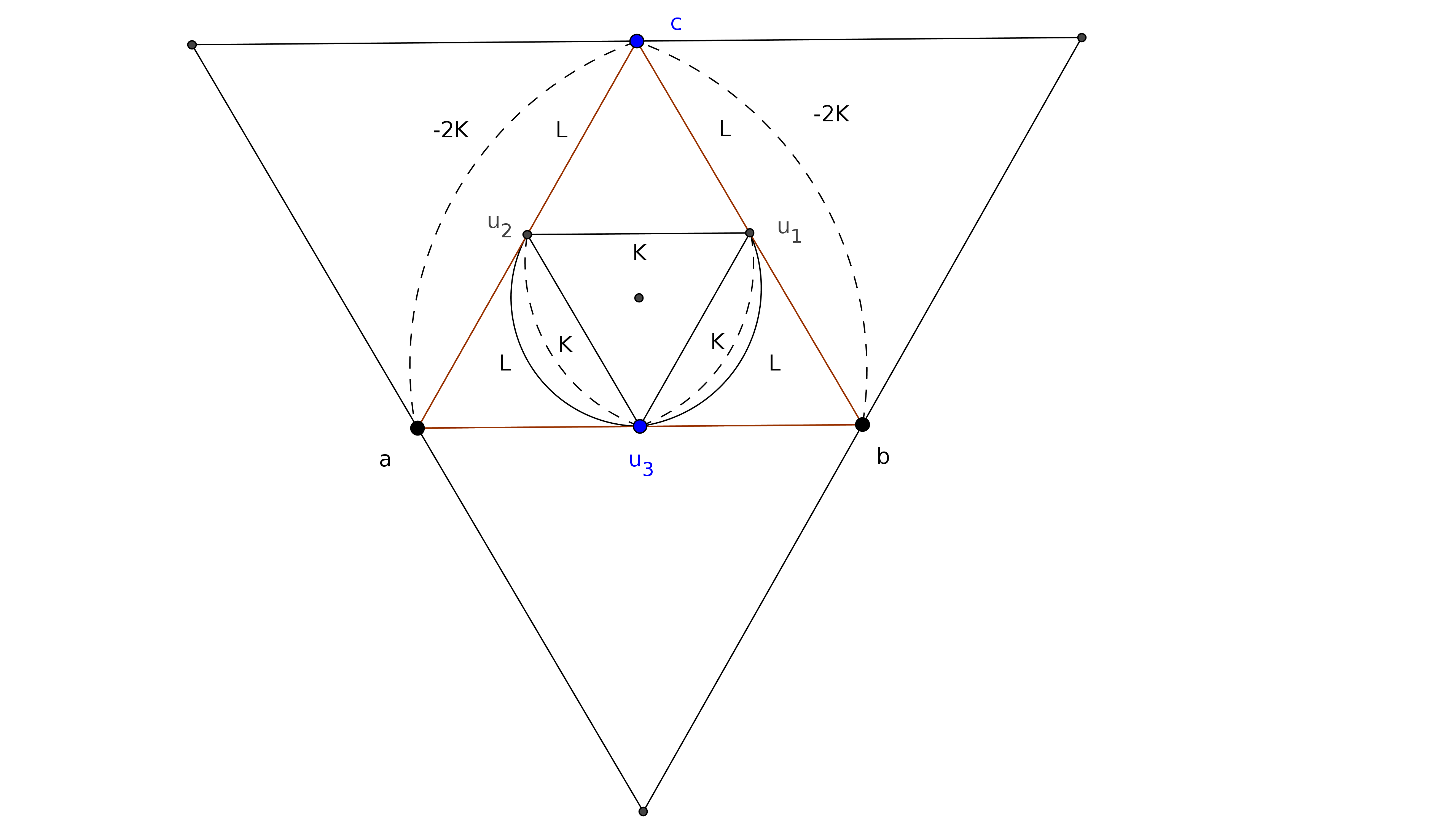}
\caption{Position of convex bodies $K \subset L \subset -2K$, with $c, u_3$ as contact points of the boundaries of $L$ and $-2K$. Part of the boundaries of $K$ and $-2K$, that are not necessarily a subset of the boundary of $L$, are marked with a dotted line.}
\label{gj3k1}
\end{figure}

Our aim is to apply an affine transformation $T$ on $L$ such that the resulting affine image $L'=T(L)$ will still satisfy $K \subset L' \subset -2K$ but $ 0 \not \in \conv (\partial (L')^{\circ} \cap \partial (-2K)^{\circ})$. According to Lemma \ref{convex}, this will give us that $d_G(L, K) < 2$, which yields the desired contradiction. An affine image $L'$ will be a small perturbation of $L$ and is defined as follows. We consider $u_3$ as the center of the coordinate system and we perform a linear transformation $T$ on $L$ defined as 
$$T(x, y) = (f(\varepsilon)x, (1-\varepsilon)y),$$
where $\varepsilon>0$ and the factor $f(\varepsilon)>1$ is such that the points $u_1, u_2$ are mapped \textbf{outside} the triangle $abc$. Thus, $T$ shrinks $L$ in the vertical direction and elongates it in the horizontal direction (see Figure \ref{gj3k1trans}). Note that there is some positive distance between $u_2$ and $\partial (-2K)$ and between $u_3$ and $\partial (-2K)$. There is also some positive distance between $\partial L$ and $a$ and between $\partial L$ and $b$ as we have assumed that $a, b$ do not belong to $L$. Therefore, by choosing $\varepsilon$ small enough, we can guarantee that the image $L'$ of $L$ will still be contained in $-2K$. Moreover, the part of $L$ that is below the horizontal line $u_1u_2$ will be covered by $L'$. As $K$ does not have any points above this line we also have $K \subset L'$. In this way, we have reduced the contact point $c$ and therefore the origin does not belong the convex hull of the set $\partial (L')^{\circ} \cap \partial (-2K)^{\circ}$, contained now in the interior of the segment $[a, b]$. This gives a desired contradiction and conclusion follows.

\begin{figure}[h]
\centering
\includegraphics[width=\textwidth]{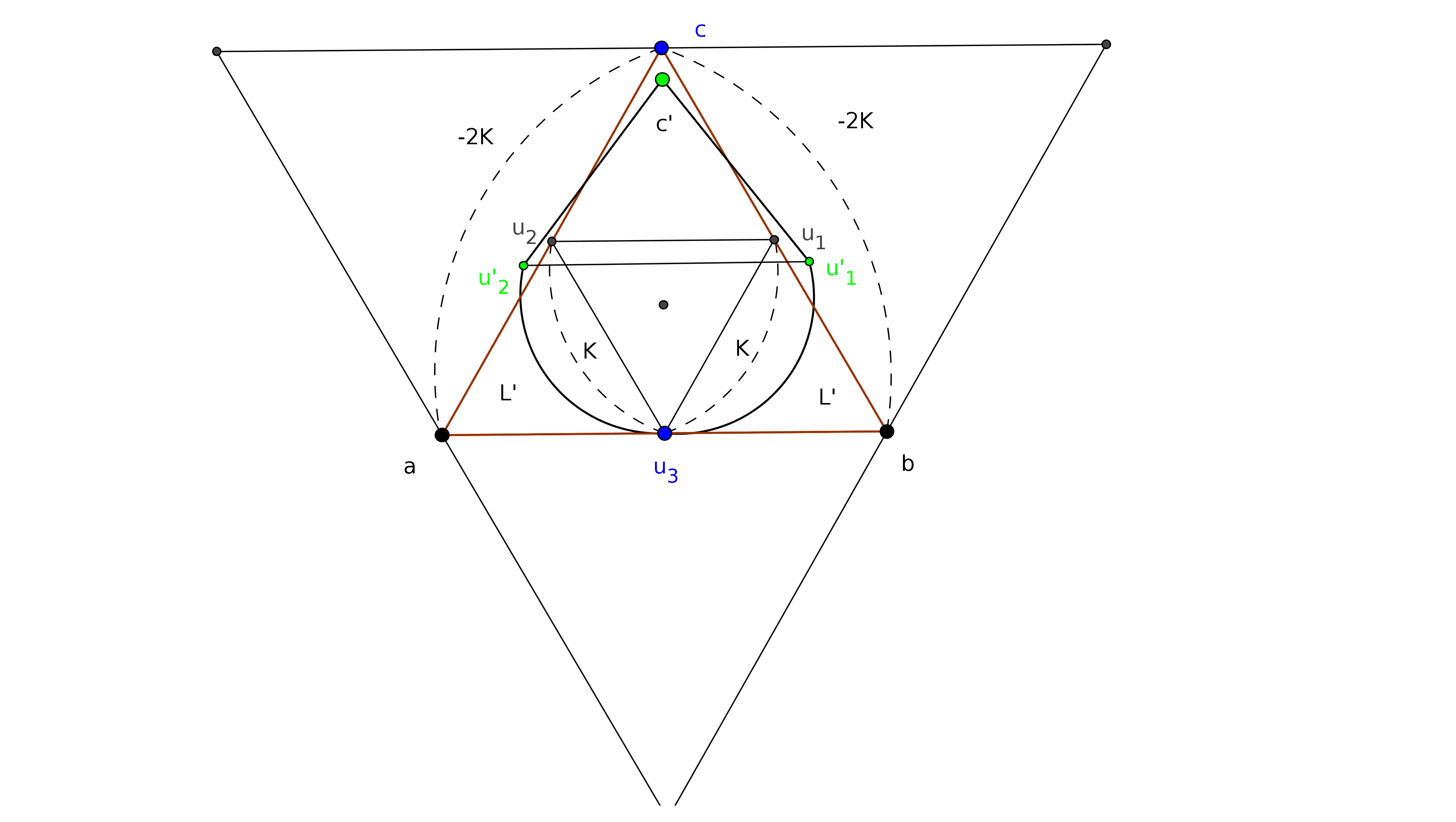}
\caption{An affine transformation $L'$ of $L$ such that $K \subset L' \subset -2K$ and all contact points of the boundaries of $L'$ and $-2K$ are in the interior of the segment $[a, b]$.}
\label{gj3k1trans}
\end{figure}

Now, we consider the case in which one of the segments $[u_2, u_3]$, $[u_1, u_3]$ is contained in the boundary of $K$. Without loss of generality, we can suppose that $[u_2, u_3] \subset \partial K$. Thus, the boundary of $K$ is a union of the line segments $[u_1, u_2]$, $[u_2, u_3]$ and some convex curve between $u_1$ and $u_3$. Similarly, the boundary of $-2K$ is a union of the line segments $[a, b]$, $[b, c]$ and some convex curve between $a$ and $c$. The contact points of the boundaries of $L$ and $-2K$ are still all contained in the interior of the side $[a, b]$, with the exception of the point $c$. See Figure \ref{gj3k1odcinek}.

\begin{figure}[h]
\centering
\includegraphics[width=\textwidth]{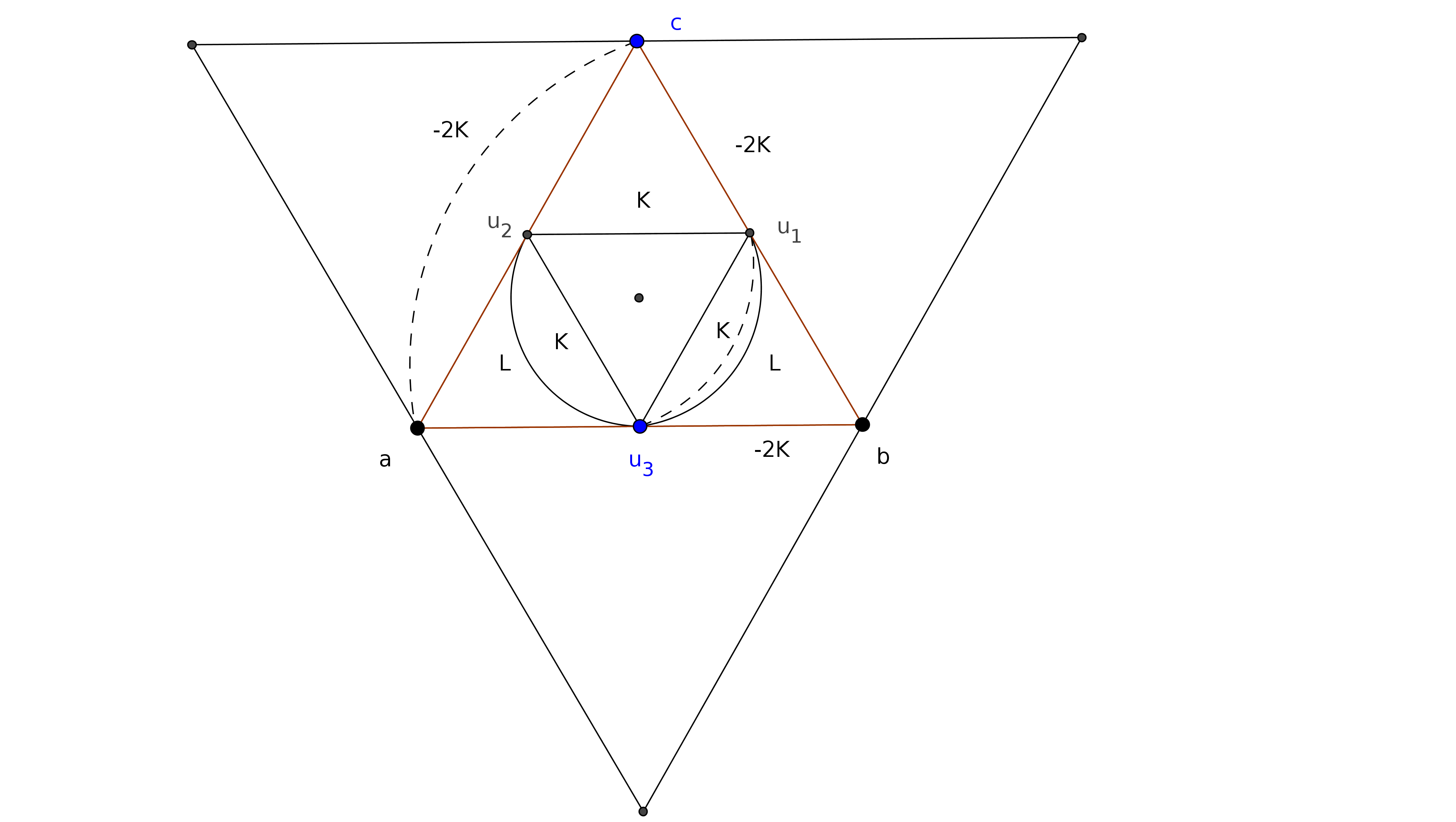}
\caption{Position of convex bodies $K \subset L \subset -2K$. Now the segment $[u_2, u_3]$ is in contained in the boundary of $K$ and the segment $[b, c]$ is contained in the boundary of $-2K$.}
\label{gj3k1odcinek}
\end{figure}

We want to use similar argument as before, but we are constrained by the fact that now there is no space between $[b, c]$ and the boundary of $-2K$. We begin therefore by shifting both $K$ and $L$ to the left along the horizontal direction (while keeping $-2K$ stationary) using the map: $(x, y) \mapsto (x-\varepsilon, y)$. We obtain convex bodies $K' \subset L'$ (see Figure \ref{gj3k1odcinektrans}).

\begin{figure}[h]
\centering
\includegraphics[height=9 cm]{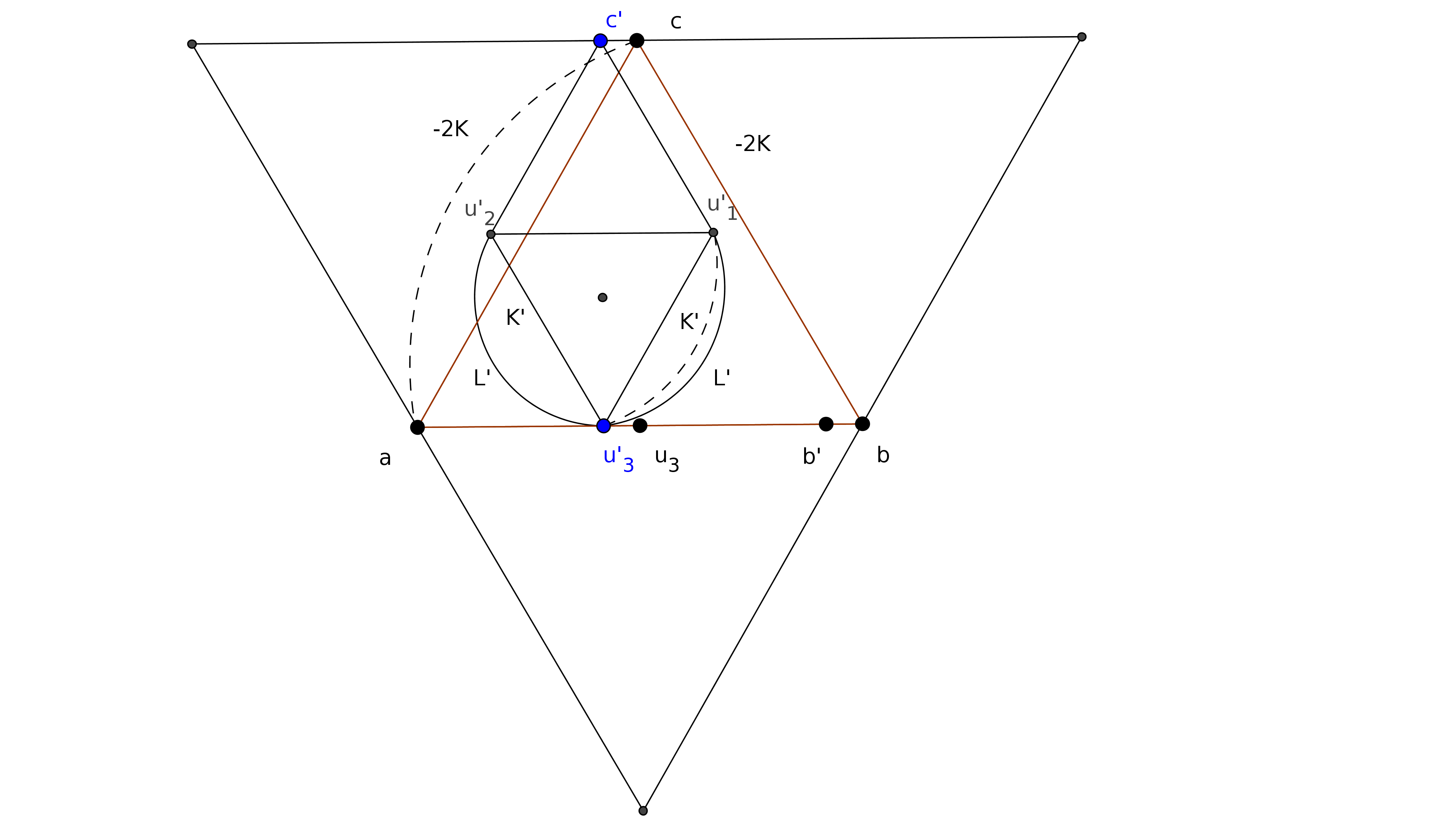}
\caption{Convex bodies $K' \subset L'$ are shifts of $K$ and $L$ by $\varepsilon$ to the left.}
\label{gj3k1odcinektrans}
\end{figure}

Note that $K' \subset -2K$ for a sufficiently small $\varepsilon$. Moreover, the part of $L'$ that is below $u_1'u_2'$ will also be contained in $-2K$ for a small $\varepsilon$, as there is some positive distance from $\partial (-2K)$ to the part of $L'$ between $u_2'$ and $u_3'$. Triangle $u_1'u_2'c'$ may not be contained in $-2K$, as $c'$ can be outside. We want to find an affine image $L''=T(L')$ of $L'$ such that $K' \subset L'' \subset -2K$. We consider $u_3'$ as a center of the coordinate system and we suppose that the segment $[u_1, u_2]$ is of the unit length. We perform a linear transformation $T$ defined as
$$T(x, y) = \left(\frac{1-\varepsilon}{1-2\varepsilon}x, (1-\varepsilon) y \right).$$
\begin{figure}[h]
\centering
\includegraphics[width=\textwidth]{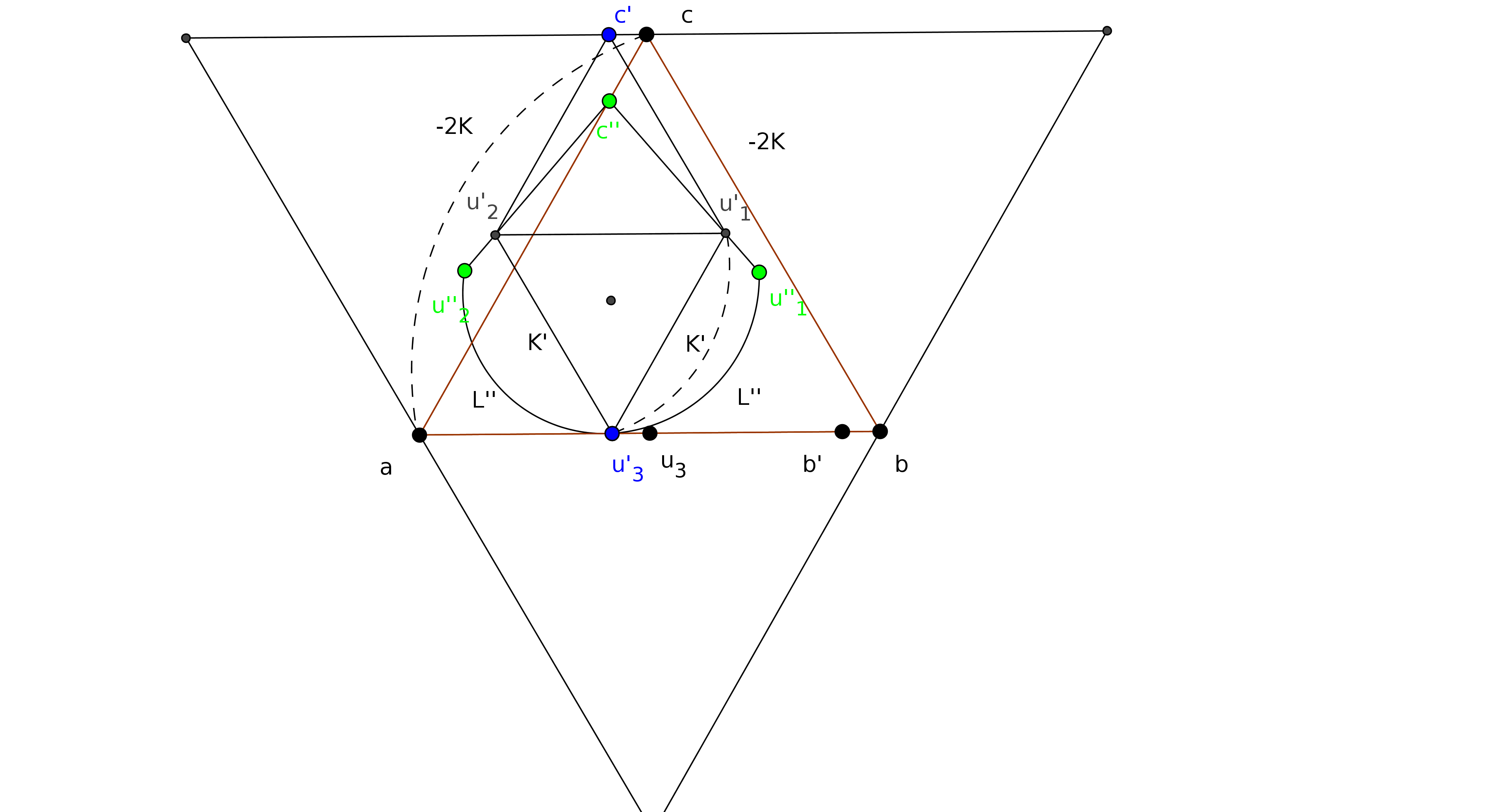}
\caption{An affine transformation $L''$ of $L'$ such that $K' \subset L'' \subset -2K$ and all contact points of the boundaries of $L''$ and $-2K$ are in the interior of the segment $[a, b]$.}
\label{gj3k1odcinektrans2}
\end{figure}
By a straightforward calculation, one can easily check that $T(c')$ is an interior point of the segment $[a, c]$ and that the triples of points $(T(c'), u_1', T(u_1'))$, $(T(c'), u_2', T(u_2'))$ are collinear (see Figure \ref{gj3k1odcinektrans2}). For sufficiently small $\varepsilon>0$, the convex body $L''$ will not meet the boundary of $-2K$ on the left. We started however, by shifting $L$ to the left by $\varepsilon$ and then elongating in the horizontal direction. Therefore, we need to check carefully that $L''$ and the boundary of $-2K$ will not meet to the right for an appropriate $\varepsilon$. We shall do it by simple means of the analytic geometry. We have
$$u_3'=(0, 0),\ u_3 = (\varepsilon, 0),\ b=(1+\varepsilon, 0),\ b' = (1, 0),\ c=(\varepsilon, \sqrt{3}),$$
$$c'=(0, \sqrt{3}),\ u_1'= \left ( \frac{1}{2}, \frac{\sqrt{3}}{2} \right ),\ u_1=\left (\frac{1}{2} + \varepsilon, \frac{\sqrt{3}}{2} \right).$$
We need to check that $T(L' \cap \conv\{ u_3', b', u_1'\}) \cap \partial (-2K) \subset [u_3', b]$. Since $(1, 0) \not \in L'$, we have that $\sup_{(x, y) \in L} x = r < 1$. Let us fix $(x_0, y_0) \in L' \cap \conv\{ u_3', b', u_1'\}$. Then $r \geq x_0>0$, $y_0>0$ and 
\begin{equation}
\label{linia}
\sqrt{3}x_0 + y_0 < \sqrt{3}.
\end{equation} 
Moreover, $T(x_0, y_0) = \left(\frac{1-\varepsilon}{1-2\varepsilon}x_0, (1-\varepsilon) y_0 \right)$. What we need to check is that $T(x_0, y_0)$ lies to the left of the line passing through $b$ and $c$. If we take $\varepsilon < \frac{1-r}{2}$ then
$$\sqrt{3}\frac{\varepsilon}{1-2\varepsilon} r < \sqrt{3} \varepsilon,$$
which, combined with the inequality $-\varepsilon y_0 \leq 0$, yields
$$\sqrt{3}\frac{\varepsilon}{1-2\varepsilon} x_0 - \varepsilon y_0 \leq \sqrt{3} \varepsilon.$$
Adding this with the inequality (\ref{linia}) we obtain
$$\frac{1-\varepsilon}{1-2\varepsilon} \sqrt{3}x_0 + (1-\varepsilon)y_0 \leq \sqrt{3}(1+\varepsilon),$$
which shows that $T(x_0, y_0)$ lies on the left side of the line passing through $b$ and $c$. Thus, we have $K' \subset L'' \subset -2K$ and all contact points $\partial L'' \cap \partial (-2K)$ are contained in the interior of the segment $[a, b]$. Again, this is a contradiction with Lemma \ref{convex}.

\textbf{Case (1c)}: there are more contact points of $\partial L$ and $\partial (-2K)$ in the set $\{a, b, c\}$. If all of $a, b, c$ are contact points, then clearly $L = \conv \{a, b, c\}$. We are therefore left with the case where exactly two of $a, b, c$ are contact points. Let them be $b$ and $c$. Then $L$ contains the trapezoid with vertices $b, c, u_2, u_3$ and has possibly some additional part contained in the triangle $u_2au_3$. In other words, the boundary of $L$ is a union of segments $[u_3, b], [b, c], [c, u_2]$ and some convex curve connecting $u_2$ with $u_3$. As $b, u_1, c \in \partial (-2K)$ the whole segment $[b, c]$ is contained in the boundary of $-2K$. Moreover, there are possibly some parts of $-2K$ contained in the triangles $ayc$ and $abz$. See Figure \ref{gj3k2} for an illustration of this configuration.

\begin{figure}[h]
\centering
\includegraphics[width=\textwidth]{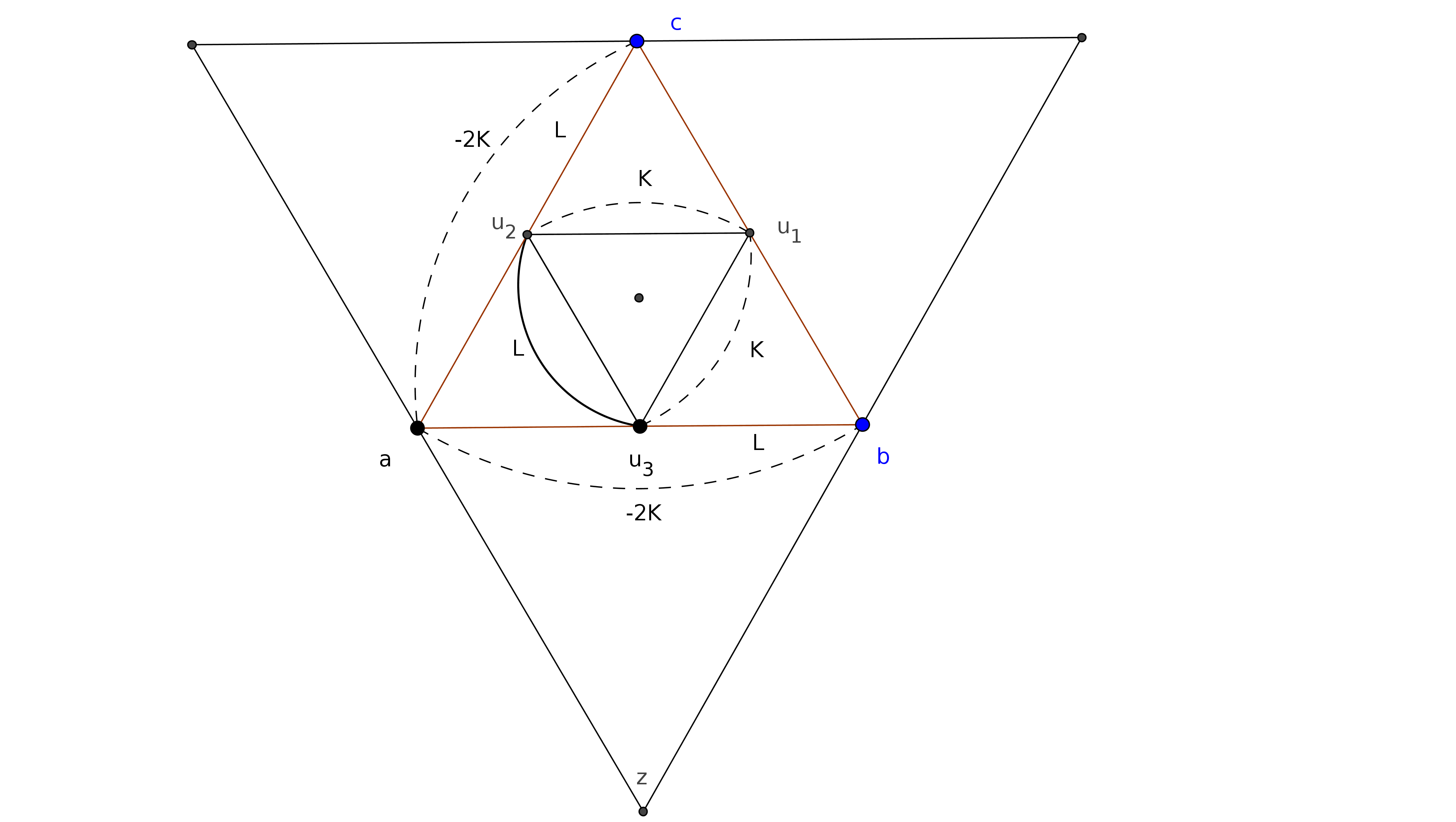}
\caption{Position of convex bodies $K \subset L \subset -2K$, with $b, c$ as contact points of the boundaries of $L$ and $-2K$. Part of the boundaries of $K$ and $-2K$, that are not necessarily a subset of the boundary of $L$, are marked with a dotted line. Part of the boundary of $L$ between $u_2$ and $u_3$ is some convex curve.}
\label{gj3k2}
\end{figure}

As $a \not \in \partial L^{\circ} \cap \partial (-2K)^{\circ}$ and yet $0 \in \conv \left ( \partial L^{\circ} \cap \partial (-2K)^{\circ} \right )$ by Lemma \ref{convex}, it is clear that there are two parallel lines that are supporting $-2K$ and $L$ in $b$ and $c$. Let us call this supporting line in $b$ by $\ell$ and let us also denote $z=4u_3$. Line $\ell$ is contained in the angle $\angle abz$. If it coincides with the line $ab$ then the segment $[a, b]$ is contained in the boundary of $-2K$ and similarly, if $\ell$ coincides with the line $bz$ then the segment $[a, c]$ is contained in the boundary of $-2K$. If both of these segments are in the boundary of $-2K$ then obviously $-2K = \conv \{a, b, c \}$. Roles played by $b$ and $c$ are now symmetric and therefore, without losing the generality, we can assume that $u_2 \not \in \partial (-2K)$. In particular, the line $\ell$ does not coincide with the line $bz$ and there is some positive distance between $u_2$ and $\partial(-2K)$. There might be no positive distance between $u_3$ and $\partial(-2K)$, if $\ell$ coincides with the line $ab$.

Again, we will find and affine image $L'$ of $L$ such that $K \subset L' \subset -2K$ and contact points are reduced. Note that an affine image of the trapezoid $bcu_2u_3$ is an arbitrary trapezoid with the ratio of the bases equal to $2$. Consider an affine mapping $T: \mathbb{R}^2 \to \mathbb{R}^2$ such that:
\begin{itemize}
\item $T(b)=b$, 
\item $T(c)=c'$, where $c' \in [b, c]$ and the segment $[c, c']$ is of length $\varepsilon>0$,
\item $T(u_2)=u_2'$ where $u_2'$ is outside of the triangle $abc$ and the points $u_2', u_2, c'$ are collinear,
\item $T(u_3)=u_3'$, where $u_3'$ lies in the interior of the segment $[a, c]$.
\end{itemize}

\begin{figure}[h]
\centering
\includegraphics[width=\textwidth]{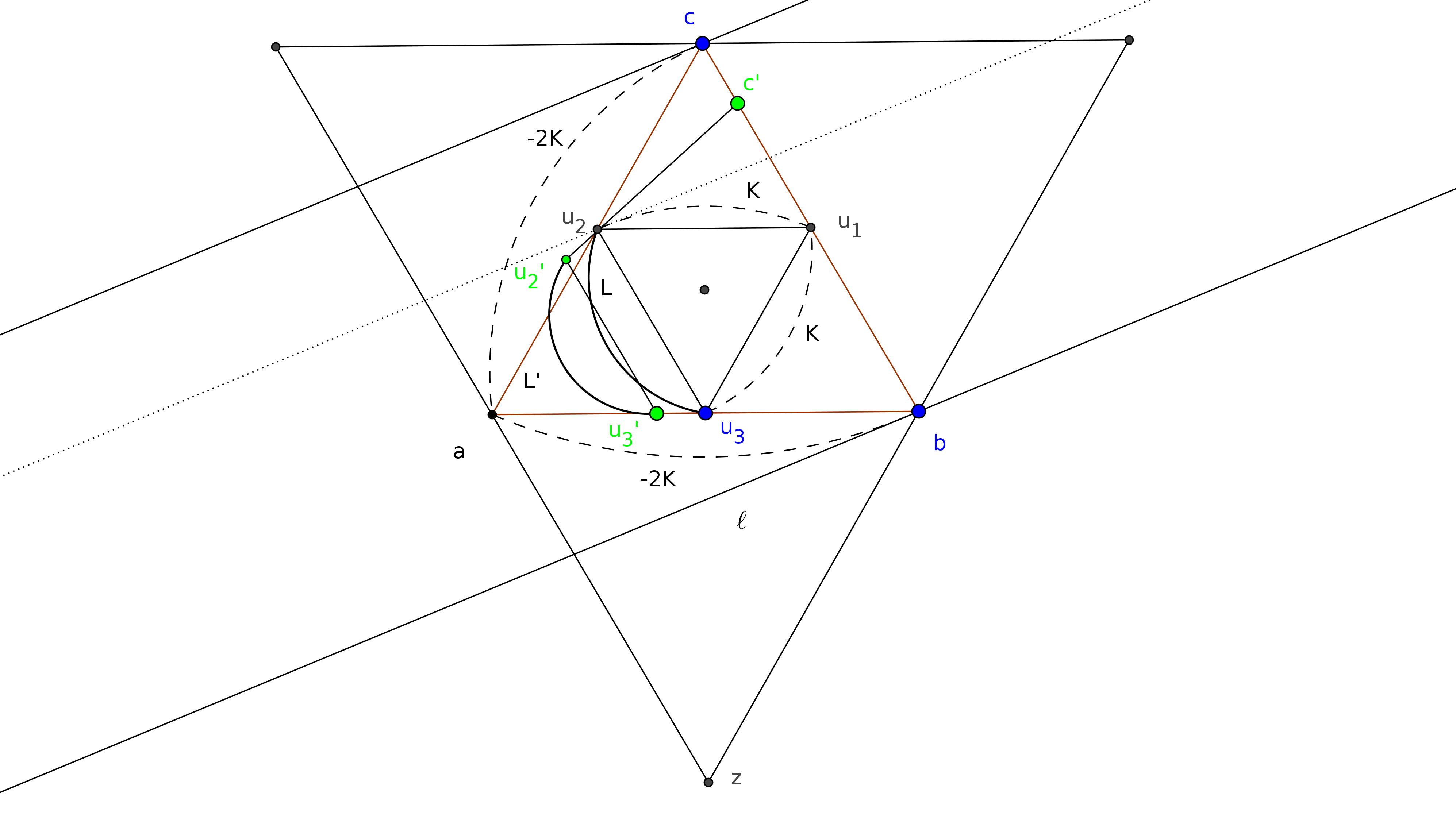}
\caption{An affine transformation $L'$ of $L$ such that $K \subset L' \subset -2K$ and all contact points of the boundaries of $L'$ and $-2K$ are in the half-open segment $[b, a)$.}
\label{gj3k2trans}
\end{figure}

The affine map $T$ shrinks the trapezoid in the direction determined by the bases and elongates it in the perpendicular direction (see Figure \ref{gj3k2trans}). For a sufficiently small $\varepsilon>0$, the part of $L'$ between $u_2'$ and $u_3'$ is still contained in $-2K$. Thus, $L' \subset -2K$ and we are left with proving the other inclusion. The triangles $u_2u_1u_3$ and $u_1u_3b$ are still contained in $L'$. Moreover, for $\varepsilon$ approaching $0$, the line $u'_2u_2$ approaches the line $ac$. Hence, for some small $\varepsilon$, this line is contained in the angle determined by the ray $u_2c$ and the ray parallel to $\ell$, going out from $u_2$. It means that the line $u'_2u_2$ is supporting to $K$ in $u_2$, which proves that $K \subset L'$. All contact points of the boundaries of $L'$ and $-2K$ are now in the half-open segment $[b, a)$. According to Lemma \ref{convex}, this gives us that $d_G(K, L) < 2$, which contradicts our assumption. In this way, we have handled all different possibilities in the situation of three contact points in the John's position of $K \subset L$ and the case $(1)$ is proved.

Assume now that $d_G(K, L)=2$ and there are $4$ contact pairs $(u_i, v_i)_{1 \leq i \leq 4}$ in the John's position $K \subset L$. Let us assume that $u_1, u_2, u_3, u_4$ are vertices of a convex quadrilateral in this order. We start with some general observations. According to Lemma \ref{convex}, we have that $0 \in \conv \left ( \partial L^{\circ} \cap \partial (-2K)^{\circ} \right ).$ Therefore, there are at least two contact points of the boundaries $L$ and $-2K$. Let $x \in \partial L \cap \partial (-2K)$. Let $w \in K^{\circ}$ be such that $\langle x, w \rangle = -2$. From condition (\ref{convzalozenie}) it follows that both sets $A_x$ and $B_x$, defined as in Lemma \ref{conditions}, are $2$-element sets. Suppose that $A_x=\{u_i, u_j\}$ and $B_x=\{u_k, u_l\}$. According to Lemma \ref{conditions}, we have
$$\frac{-x}{2} \in \conv\{u_k, u_l\}, \: \langle w, u_k \rangle = \langle w, u_l \rangle = 1, \: \langle x, v_i \rangle = \langle x, v_j \rangle = 1.$$

In particular, $x$ is the intersection of the lines determined by $v_i$ and $v_j$. Therefore, if $x' \in \partial L \cap \partial (-2K)$ and $x' \neq x$, then $A_{x'} \neq A_{x}$ and $B_{x'} \neq B_{x}$. Moreover, it is clear that $w$ is parallel to the line determined by $u_k, u_l$. Segments $[x, u_i], [x, u_j]$ are contained in the boundary of $L$. We claim that also the segment $[u_k, u_l]$ is contained in the boundary of $K$. Indeed, if $-\frac{x}{2}$ is an interior point of the segment $[u_k, u_l]$ then our claim is obvious. If $x = -2u_k$ or $x=-2u_l$, then by the previous observations, the line determined by $u_k$, $u_l$ is supporting to $K$ at $u_k$ or $u_l$ and we arrive at the same conclusion.

The lines determined by $v_1, v_2, v_3, v_4$ have at most six intersection points, but no five of them can lie in a convex position. Therefore, there are at most four contact points of the boundaries of $L$ and $-2K$. Let us start with the case, in which there are exactly four of them.

\textbf{Case (2a)}: there are exactly four points of contact of $\partial L$ and $\partial (-2K)$. Supppose that $\partial L^{\circ} \cap \partial (-2K)^{\circ} = \{x_1, x_2, x_3, x_4 \}$. Note that for each $1 \leq i \leq 4$ vector $u_i$ belongs to at most $2$ different sets $A_{x_j}$. Indeed, let us assume that for example $u_1 \in A_{x_1}, A_{x_2}, A_{x_3}$. Then
$$1 = \langle x_1, v_1 \rangle = \langle x_2, v_1 \rangle = \langle x_3, v_1 \rangle,$$
so that the points $x_1, x_2, x_3$ are collinear. Thus, the whole segment $[x_1, x_2]$ is contained in the boundaries of $L$ and $-2K$, which contradicts the fact that there are only $4$ contact points. See Figure \ref{gj4k4}.

\begin{figure}[h]
\centering
\includegraphics[width=\textwidth]{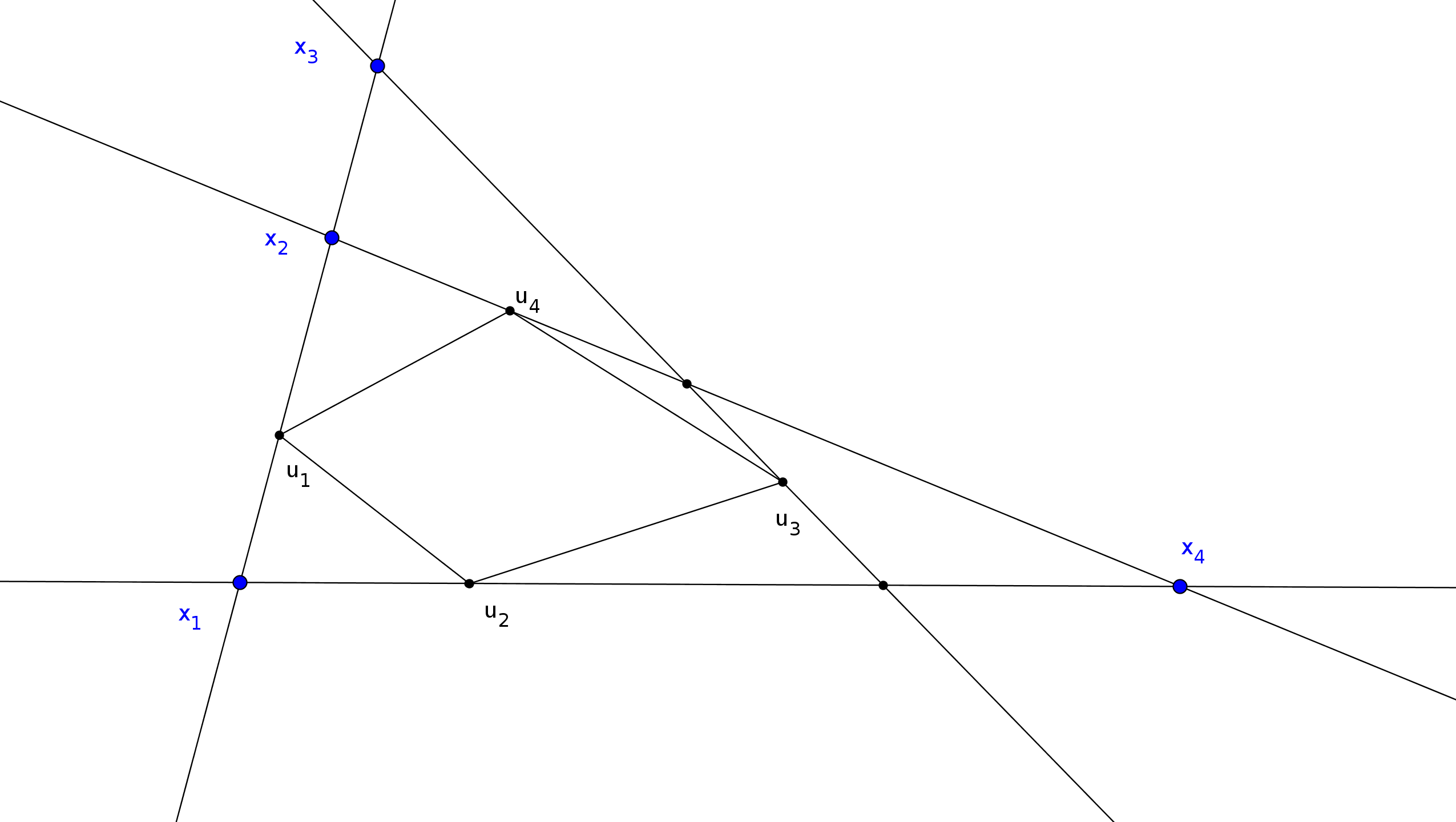}
\caption{Position of convex bodies $K \subset L \subset -2K$, with $x_1, x_2, x_3, x_4$ as contact points of the boundaries of $L$ and $-2K$. Assumption that $x_1, x_2, x_3$ lie on one line easily leads to a contradiction.}
\label{gj4k4}
\end{figure}

Therefore, for each $1 \leq i \leq 4$ we have that $u_i$ belongs to exactly two of the sets $A_{x_1}, A_{x_2}, A_{x_3}, A_{x_4}$. But, if for example $u_1 \in A_{x_1}, A_{x_2}$, then by the previous observation the segments $[x_1, u_1]$ and $[u_1, x_2]$ are contained in the boundary of $L$ and thus $[x_1, x_2] \subset \partial L$. Hence, it is clear that $L = \conv \{x_1, x_2, x_3, x_4\}$. Moreover, if $u_1, u_2 \in A_{x_1}$, then $[u_3, u_4] \subset \partial K$ and for the same reason $K = \conv \{u_1, u_2, u_3, u_4\}$. We conclude that both $K$ and $L$ are quadrilaterals. We can use now a result of Lassak (see \cite{lassak}): the Banach-Mazur distance between two convex quadrangles is at most $2$ and the equality holds if and only if one of them is the parallelogram and the other one is the triangle. So in this case, the conclusion follows.

\textbf{Case (2b)}: there are exactly two points of contact of $\partial L$ and $\partial (-2K)$. Suppose that $A_x=\{u_3, u_4 \}$ and $B_x = \{u_1, u_2 \}$. From Lemma \ref{convex} it follows that there are two parallel supporting lines of $L$ and $-2K$ passing through $x$ and $y$. From previous observations we see that they are parallel to the line $u_1u_2$. It follows that the line determined by two points in $B_y$ is parallel to $u_1u_2$. Since $B_y \neq B_x$, we conclude that $B_y=\{u_3, u_4\}$. Quadrilateral with vertices $u_1, u_2, u_3, u_4$ is therefore a trapezoid. Moreover, $[u_3, x], [u_4, x], [u_1, y], [u_2, y] \subset \partial L$ and $[u_1, u_2], [u_3, u_4] \subset \partial K$. Consider the parts of $\partial L$ between $u_1$ and $u_4$ and between $u_2$ and $u_3$. They are at some positive distance to $\partial(-2K)$ as $x, y$ are the only contact points of $\partial L$ and $\partial(-2K)$. See Figure \ref{gj4k2}.

\begin{figure}[h]
\centering
\includegraphics[width=\textwidth]{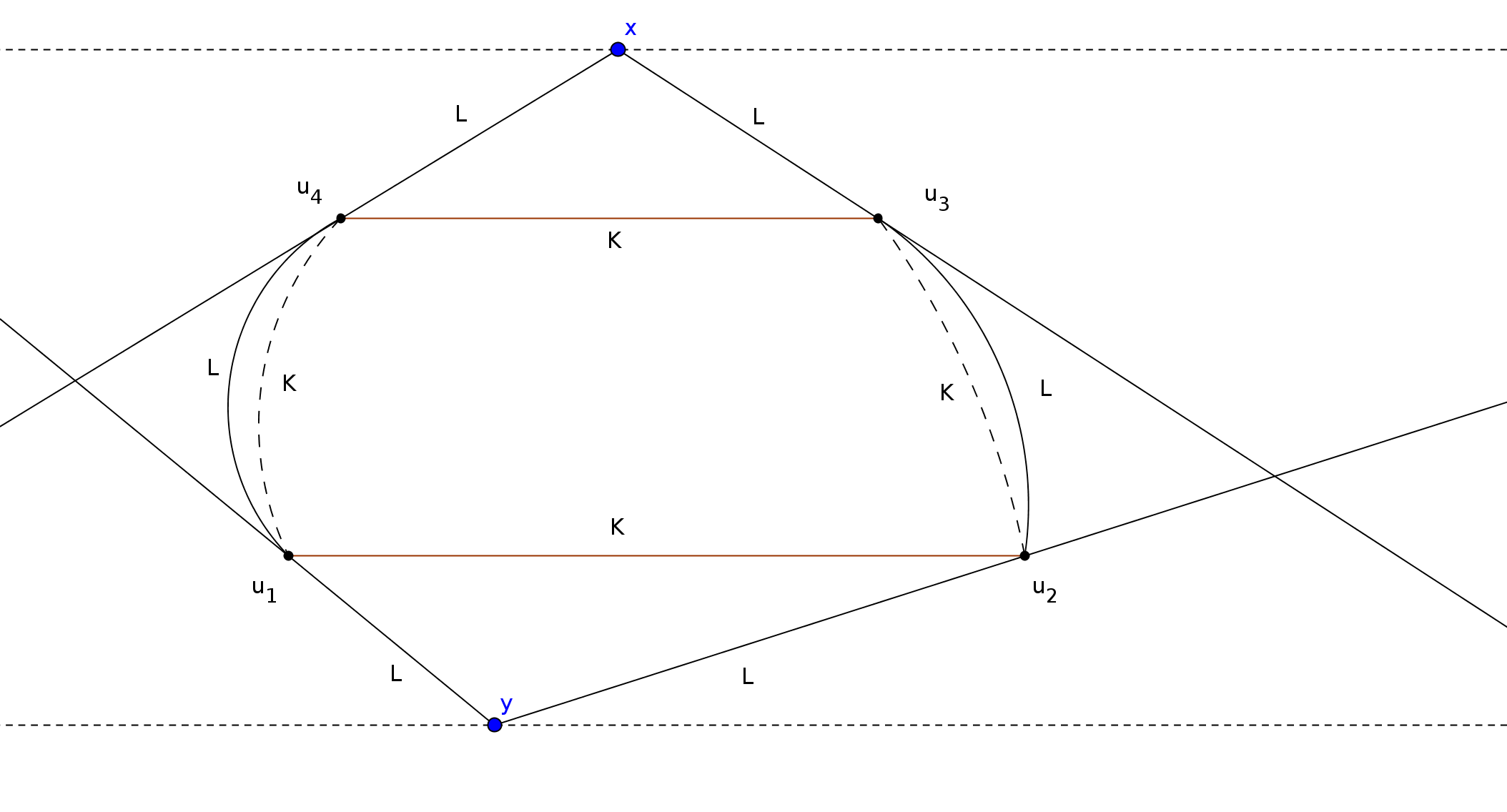}
\caption{Position of convex bodies $K \subset L \subset -2K$, with $x$ and $y$ as contact points of the boundaries of $L$ and $-2K$. In this case, the quadrilateral $u_1u_2u_3u_4$ is a trapezoid. The segments $[u_3, x], [u_4, x], [u_1, y], [u_2, y]$ are contained in the boundary of $L$ and the segments $[u_1, u_2], [u_3, u_4]$ are contained in the boundary of $K$.}
\label{gj4k2}
\end{figure}

Our aim is to find an affine copy $T(L)=L'$ of $L$ such that $K \subset L' \subset -2K$ and the number of contact points of $\partial L' \cap \partial (-2K)$ is reduced to $0$. For this purpose we consider any affine mapping  $T: \mathbb{R}^2 \to \mathbb{R}^2$ satisfying the following conditions:
\begin{itemize}
\item lines $T(u_1)T(u_2), T(u_3)T(u_4)$ are parallel to lines $u_1u_2$, $u_3u_4$ and are lying between them,
\item $T(x) \in \conv\{x, u_3, u_4\}$ and $T(y) \in \conv \{y, u_1, u_2\}$,
\item segments $[T(x), T(u_3)], [T(x), T(u_4)], [T(y), T(u_1)], [T(y), T(u_2)]$ are intersecting interiors of the segments $[x, u_3], [x, u_4], [y, u_1], [y, u_2]$ respectively,
\item segment $[x, x']$ is of the length $\varepsilon>0.$
\end{itemize}

\begin{figure}[h]
\centering
\includegraphics[width=\textwidth]{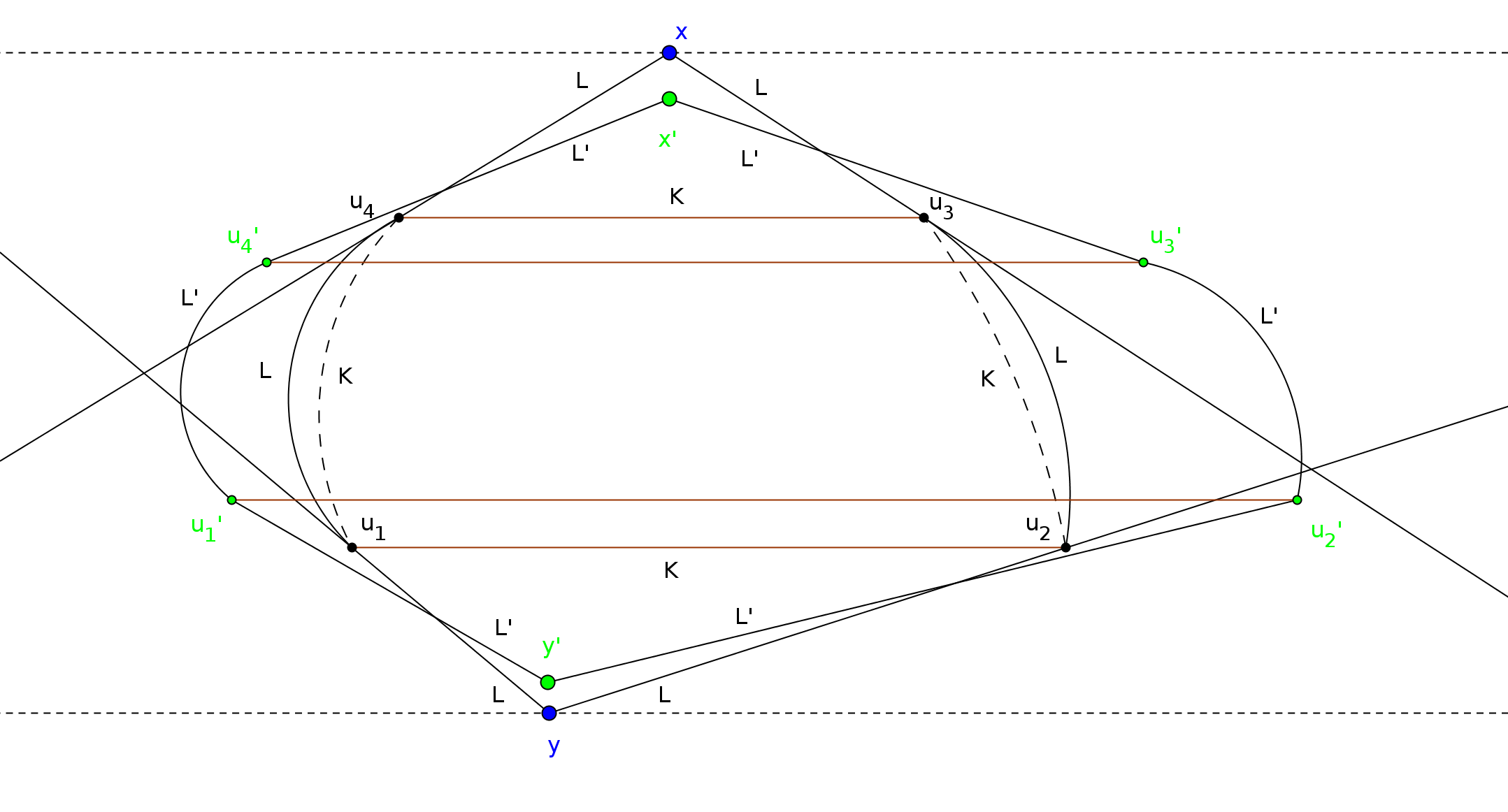}
\caption{An affine transformation $L'$ of $L$ such that $K \subset L' \subset -2K$ and there are no contact points of the boundaries of $L'$ and $-2K$.}
\label{gj4k2trans}
\end{figure}

See Figure \ref{gj4k2trans}. Similar reasoning as we have used before, shows that for a sufficiently small $\varepsilon>0$ such an affine mapping $T$ exists and also $K \subset L' \subset -2K$. Moreover, $\partial L' \cap \partial (-2K) = \emptyset.$ This yields a desired contradiction in this case. In this way, we have handled the case $(2)$.

\textbf{Case (3)}: there are three contact points between $\partial L$ and $\partial(-2K)$. Unfortunately, in this case, we have to assume additionally that $L$ is symmetric. But for a moment, let $L$ still be arbitrary. Let $\partial L \cap \partial (-2K) = \{ x_1,\, x_2,\ x_3\}$ and let $w_i \in K$ be such that $\langle x_i, w_i \rangle = -2$ for $1 \leq i \leq 3$. Note that $A_{x_i} \neq \{u_1, u_3\}$ for each $1 \leq i \leq 3$. Otherwise, the segment $[u_1, u_3]$ would be contained in the boundary of $K$, but it is impossible, as it is a diagonal. For the same reason $A_{x_i} \neq \{u_2, u_4\}$ for each $1 \leq i \leq 3$. Without loss of generality, we can therefore assume that $A_{x_1} = \{u_3, u_4\}$, $A_{x_2} = \{u_2, u_3\}$, $A_{x_3} = \{u_1, u_2\}$. By $x_4$ we denote the intersection of supporting lines at $u_2$ and $u_3$, that is $x_4$ satisfies $\langle x_4, v_2 \rangle = \langle x_4, v_3 \rangle = 1.$

The segments $[u_1, u_2]$, $[u_2, u_3]$ and $[u_3, u_4]$ are contained in $\partial K$, and the remaining part of the $\partial K$ is a convex curve contained in the triangle $\conv\{u_1, x_2, u_4\}$. On the other hand, $\partial L$ contains the segments $[x_1, x_2]$, $[x_2, x_3]$, $[x_3, u_3]$, $[x_1, u_2]$ and the remaining part of the boundary is a convex curve contained in the triangle $\conv \{u_2, u_3, x_4\}$. It is clear that $x_1 \neq -2u_3$ and $x_1 \neq -2u_4$, as we have already observed that there is a unique vector $w_1 \in K$ such that $\langle x_1, w_1 \rangle = -2$, which is parallel to $u_3u_4$ and is, in particular, different from $v_3$ and from $v_4$. For the same reason $x_2 \neq -2u_2$, $x_2 \neq -2u_3$, $x_3 \neq -2u_1$ and $x_3 \neq -2u_2$. It means that $x_1, x_2, x_3$ lie in the interiors of the segments contained in the boundary of $-2K$. See Figure \ref{gj4k3} for an illustration of this configuration.

\begin{figure}[h]
\centering
\includegraphics[width=\textwidth]{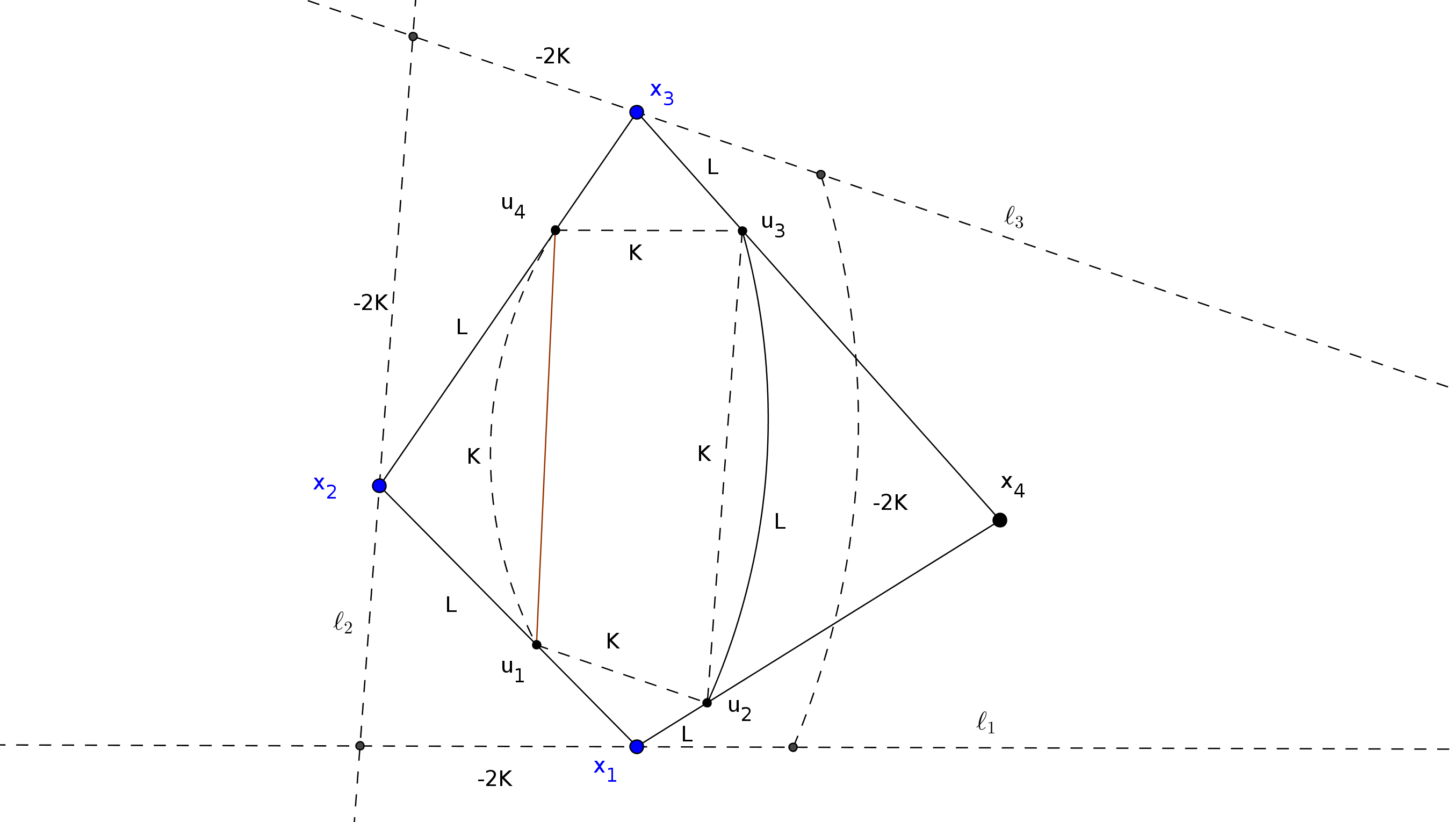}
\caption{Position of convex bodies $K \subset L \subset -2K$, with $x_1$, $x_2$, $x_3$ as contact points of the boundaries of $L$ and $-2K$. The segments $[x_1, x_2]$, $[x_2, x_3]$, $[x_3, u_3]$, $[x_1, u_2]$ are contained in the boundary of $L$ and the segments $[u_1, u_2]$, $[u_2, u_3]$ and $[u_3, u_4]$ are contained in the boundary of $K$. The boundaries of $K$ and $-2K$ are marked with a dotted line.}
\label{gj4k3}
\end{figure}

Let $\ell_1, \ell_2, \ell_3$ be supporting lines to $L$ and $-2K$ at the points $x_1, x_2, x_3$ respectively. That is, $\ell_1 \parallel u_3u_4$, $\ell_2 \parallel u_2u_3$, $\ell_3 \parallel u_1u_2$. Suppose first that $\ell_1, \ell_3$ are parallel. The quadrilateral $u_1u_2u_3u_4$ is in this case a trapezoid. We shift both $K$ and $L$ by a small vector in the direction $\ell_1$ and away from $\ell_2$, obtaining convex bodies $K'$ and $L'$. Because there is a positive distance between part of $\partial L$ contained in $\conv \{u_2, u_3, x_4\}$ and the boundary of $-2K$, for a shift that is small enough we will have $K' \subset L' \subset -2K$ with $x_1, x_3$ as contact points of the boundaries of $L'$ and $\partial(-2K)$. From this moment, we can use exactly the same argument as in the case (2a) of two points of contact between the boundaries of $L$ and $-2K$. We obtain an affine copy $T(L')$ such that both contact points $x_1, x_3$ disappear and this is again a contradiction with Lemma \ref{convex} combined with the assumption $d_G(K, L)=2$. In other words, the case $\ell_1 \parallel \ell_3$ reduces to the previous one (see Figure \ref{gj4k3rownolegle}). 

\begin{figure}[h]
\centering
\includegraphics[width=\textwidth]{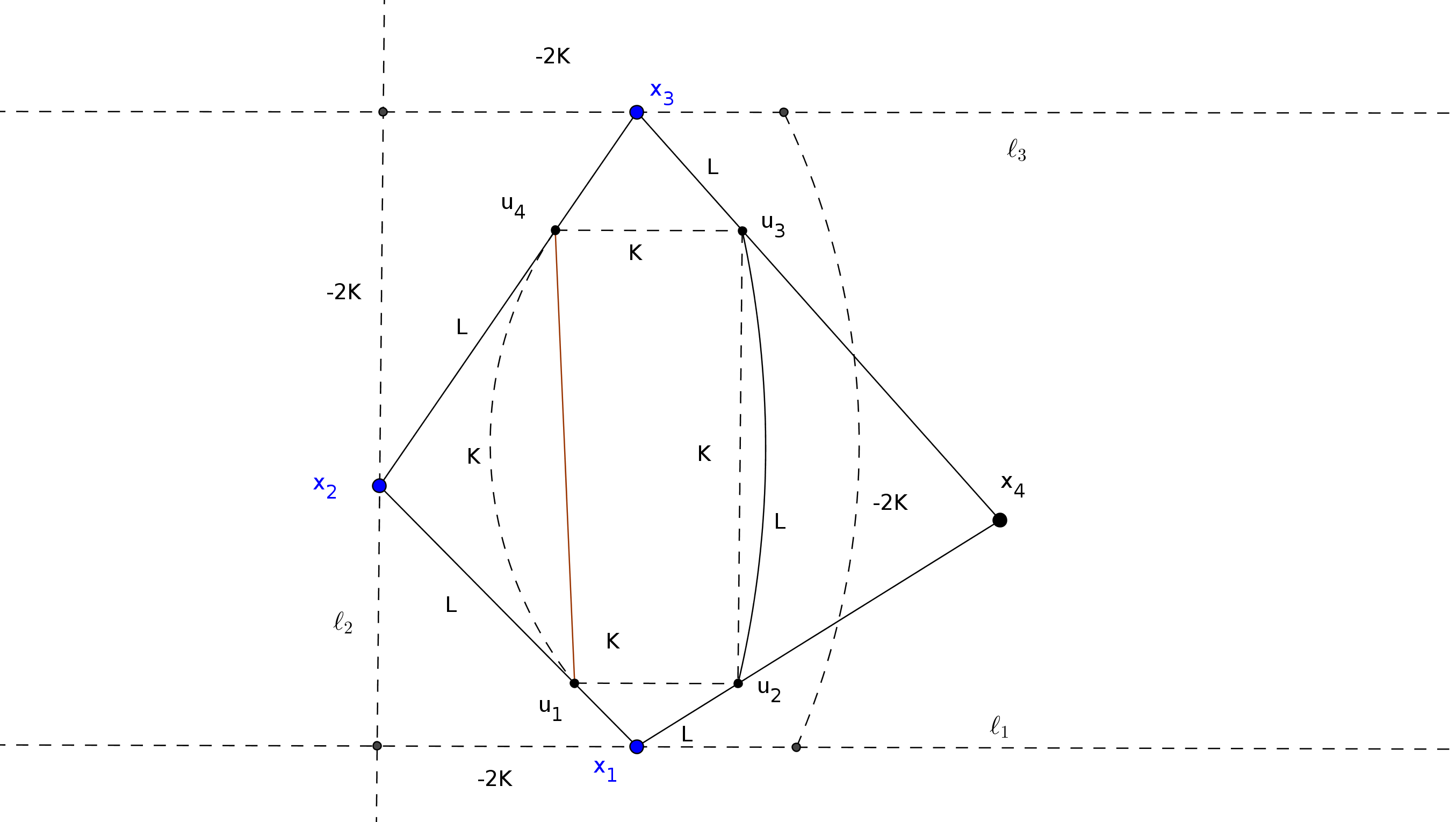}
\caption{Position of convex bodies $K \subset L \subset -2K$, with $x_1$, $x_2$, $x_3$ as contact points of the boundaries of $L$ and $-2K$ and $\ell_1 \parallel \ell_3$. By performing a small shift to the right along the $\ell_1$ we can reduce the contact point $x_2$ and we are able to apply the same reasoning as in the previous case.}
\label{gj4k3rownolegle}
\end{figure}

If the intersection of lines $\ell_3$ and $\ell_3$ lies in the other side of $\ell_2$ than $K$ and $L$, the argument gets even easier. In this case, after performing a small shift of $K$ and $L$ in the direction determined by $\ell_1$, the points $x_2, x_3$ will move away from the boundary of $-2K$. Thus, the number of contact points of $L$ and $-2K$ is reduced to one, which again gives a contradiction with Lemma \ref{convex} (see Figure \ref{gj4k3latwy}).

\begin{figure}[h]
\centering
\includegraphics[width=\textwidth]{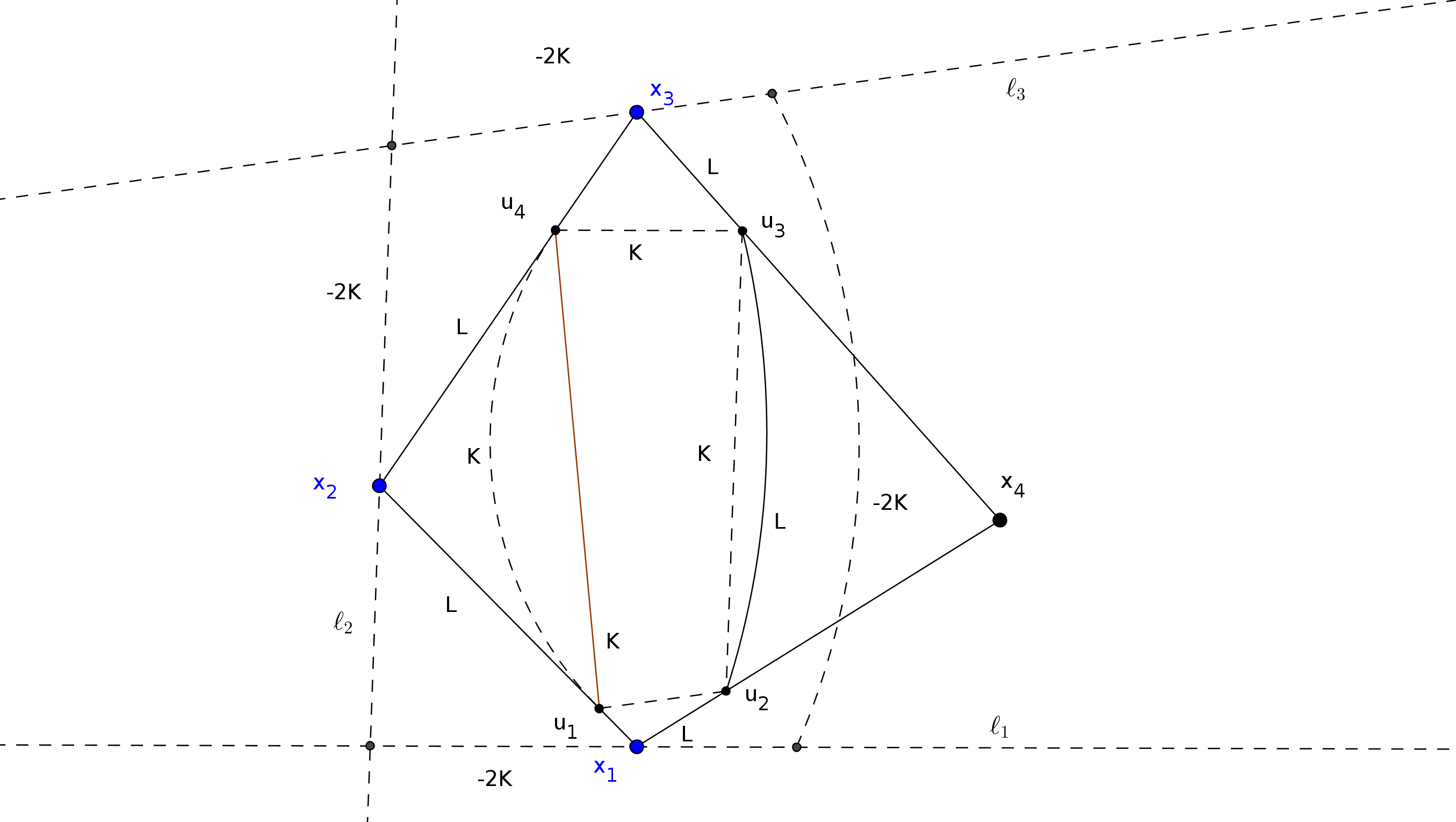}
\caption{Position of convex bodies $K \subset L \subset -2K$, with $x_1$, $x_2$, $x_3$ as contact points of the boundaries of $L$ and $-2K$ and with the intersection of $\ell_1$ and $\ell_3$ to the left of $\ell_2$. By performing a small shift to the right along the $\ell_1$, we reduce contact points $x_2, x_3$ and we reach a contradiction immediately.}
\label{gj4k3latwy}
\end{figure}

Therefore, the only interesting possibility is when $\ell_1$ and $\ell_3$ do intersect in the same side of $\ell_2$ as $K$ and $L$ are contained. This configuration is demonstrated in Figure \ref{gj4k3}. In this situation, it is not clear how to transform $L$ to reduce the contact points. However, if we assume additionally that $L$ is symmetric, then it becomes apparent that $L$ has to be the parallelogram. Indeed, let $o$ be the center of symmetry of $L$. The reflection $2o-x_1$ belongs to the boundary of $L$ and it has a supporting line parallel to $\ell_1$. Looking at the boundary of $L$ we see that the only suitable point is $x_3$. Thus, $o$ is the midpoint of the segment $[x_1, x_3]$. Moreover, the reflections of the segments $[x_1, x_2]$ and $[x_2, x_3]$ are also in the boundary and it is clear that they have to be mapped to $[x_3, x_4]$ and $[x_1, x_4]$ respectively. Thus $L=\{x_1, x_2, x_3, x_4\}$ is the parallelogram and this case was already covered by Lassak in \cite{lassakpar}. This finishes proof of Theorem.

\qed

\begin{remark}

We note that if $S=S_2$ is the triangle in $\mathbb{R}^2$ and $K$ is a planar convex body such that $d_G(K, S) = 2$, then $K$ does not have to be symmetric. It is well-known that if $K \subset S \subset -rK + v$, then we must have $r \geq 2$ (this is equivalent to the fact that the \emph{asymmetry constant} of the triangle is equal to $2$). On the other hand, $d_G(K, S) \leq 2$ by Theorem \ref{glmp}. Thus, for any convex body $K$ such that $d(K, S) > 2$ we have $d_G(K, S) = 2$. One could take for example the regular pentagon as $K$ and then $d(K, S) = 1 + \frac{\sqrt{5}}{2} > 2$. Moreover, by using a continuous deformation (with the respect to the Banach-Mazur distance) of the regular pentagon to the triangle, one can see that we can even find a convex pentagon $K$ such that $d_G(K, S) = d(K, S) = 2$.
\end{remark}

\section{Concluding remarks}
\label{concluding}
In the proof of Theorem \ref{twglowne} we were able to show that in most situations, if $K \subset L$ are in the John's position and $d_G(K, L)=2$, then we can find a small affine perturbation $L'$ of $L$ such that $K \subset L' \subset -rK+v$ for some $r<2$. We can therefore say that, in some sense, our reasoning was local. It is not clear, if such a local reasoning is enough for the case (3) of $4$ contact pairs in the John's position of $K$ and $L$ and $3$ contact points of $L$ and $-2K$. It would be very interesting to know, if for this remaining case, the same strategy could be applied or some different approach is necessary. It is also well possible, that there are some convex bodies $K$ and $L$, both different from the triangle, such that $d_G(K, L)=2$. In such case, the counterexample could be as simple as some convex quadrilateral and convex pentagon. We also remark that if our goal was to prove the theorem only for centrally-symmetric $L$, some steps of the proof could be simplified. At this moment, our approach seems to be too complicated to be used in higher dimensions in the full generality. But it seems well possible, that for some particular cases, the idea of lowering the distance by a small affine perturbations could be applied.

There are still many open and interesting questions concerning Banach-Mazur distance, even in the planar case. One natural question that seems particularly important for understanding the structure of the Banach-Mazur compactum is the following: what is the isometry group of the Banach-Mazur compactum? This question could be asked in the symmetric case, in the general case and also for the variant of the Gr\"unbaum distance. Similar results for the Hausdorff distance and the symmetric-difference distance were established (see \cite{gruber1} and \cite{gruber2}), but the case of the Banach-Mazur distance seems much more difficult, even in the two-dimensional setting. It seems plausible that the identity and the dual map should be the only isometries in the symmetric case and the identity should be the unique isometry in the non-symmetric case and for the Gr\"unbaum distance.

\end{document}